\documentclass[12pt,sort&compress]{elsarticle}
\usepackage{amsmath}
\usepackage{amsfonts}
\usepackage{amssymb}
\usepackage{amsthm}
\usepackage{bm}
\usepackage{indentfirst}
\usepackage{graphicx}
\usepackage{subfigure}
\usepackage{array}
\usepackage{fullpage}
\usepackage{color}

\DeclareMathAlphabet{\mathsfsl}{OT1}{cmss}{m}{sl}

\newcommand{\PreserveBackslash}[1]{\let\temp=\\#1\let\\=\temp}
\newcolumntype{C}[1]{>{\PreserveBackslash\centering}p{#1}}
\newcolumntype{R}[1]{>{\PreserveBackslash\raggedleft}p{#1}}
\newcolumntype{L}[1]{>{\PreserveBackslash\raggedright}p{#1}}

\renewcommand{\vartheta}{\Theta}

\definecolor{mygreen}{rgb}{0.1,0.75,0.2}

\numberwithin{equation}{section}
\newtheorem{thm}{Theorem}[section]

\theoremstyle{definition}

\usepackage{hyperref}
\usepackage{fancyvrb}
\usepackage{fancyhdr}
\biboptions{sort&compress}

\usepackage{bm}
\usepackage{mathrsfs}
\usepackage{amsmath}
\usepackage{amsfonts}
\usepackage{amsthm}
\usepackage{color}
\usepackage{setspace}
\usepackage{float}
\usepackage{url}
\usepackage{multicol}
\usepackage{subfigure}
\usepackage{bm}
\usepackage[top=1in,bottom=1in,left=1in,right=1in]{geometry}

\begin{document}

\theoremstyle{definition}
\newtheorem{remark}{Remark}

\newcommand{\vertiii}[1]{{\left\vert\left\vert\left\vert #1
    \right\vert\right\vert\right\vert}}

\begin{frontmatter}

\title{An asymptotically compatible meshfree quadrature rule for non-local problems with applications to peridynamics}

\address[nt]{Center for Computing Research, Sandia National Laboratories, Albuquerque NM}
\address[yy]{Department of Mathematics, Lehigh University, Bethlehem, PA 18015, USA}

\author[nt]{Nathaniel Trask}\ead{natrask@sandia.gov}
\author[yy]{Huaiqian You}
\author[yy]{Yue~Yu\corref{cor1}}\ead{yuy214@lehigh.edu}
\author[nt]{Michael Parks}


\begin{abstract}
We present a meshfree quadrature rule for compactly supported non-local integro-differential equations (IDEs) with radial kernels. We apply this rule to develop a strong-form meshfree discretization of a peridynamic solid mechanics model that requires no background mesh. Existing discretizations of peridynamic models have been shown to exhibit a lack of asymptotic compatibility to the corresponding linearly elastic local solution. By posing the quadrature rule as an equality constrained least squares problem, we obtain asymptotically compatible convergence via reproducability constraints. Our approach naturally handles traction-free conditions, surface effects, and damage modeling for both static and dynamic problems. We demonstrate high-order convergence to the local theory by comparing to manufactured solutions and to cases with crack singularities for which an analytic solution is available. Finally, we verify the applicability of the approach to realistic problems by reproducing high-velocity impact results from the Kalthoff-Winkler experiments.

\end{abstract}


\end{frontmatter}

\tableofcontents

\section{Background and governing equations}
We seek solutions $u \in L^2(\Omega)$ on a simply connected compact domain $\Omega \subset \mathbb{R}^d$ with Lipschitz continuous boundary $\partial \Omega$, to general non-local IDEs of the form
\begin{equation}\label{eq:generalCase}
	\mathcal{L}_\delta[u](\mathbf{x}) = \int_{B(\mathbf{x},\delta)} K(\mathbf{x},\mathbf{y}) \left(u(\mathbf{y})-u(\mathbf{x})\right) \, d\mathbf{y} = \mathbf{f}(\mathbf{x}),
\end{equation}
where $u \in C^m(\Omega)$, $\mathbf{f}$ is data, and $K$ is a kernel which for fixed $\mathbf{x}$ is supported on the ball of radius $\delta$,  $B(\mathbf{x},\delta)$. We assume that the kernel is radial, i.e. $K(\mathbf{x},\mathbf{y}) = K(\mathbf{x} - \mathbf{y})$. 
For many applications, the kernel is singular as $\mathbf{y} \rightarrow \mathbf{x}$ and special care must be taken to handle the singularity.
Non-local problems of this form require the specification of appropriate volumetric boundary conditions to be well-posed. We assume a collar of thickness $\delta$ around $\partial \Omega$, defining $\Omega^\delta = \underset{\mathbf{\mathbf{x}} \in \Omega}\bigcup {B(\mathbf{\mathbf{x}},\delta)}$ with volumetric boundary $\partial^\delta \Omega = \Omega^\delta \setminus \Omega$.

Although the technique introduced here is applicable to general problems of this form, we focus in this work on bond-based peridynamic solid models, specifically the linearized microelastic brittle model \cite{silling_2000, macek_2007, silling_2005_2}
\begin{equation}\label{eq:LPSequation}
    \rho(\mathbf{x}) \frac{d^2}{dt^2} \mathbf{u}(\mathbf{x}) = \int_{B(\mathbf{x},\delta)} c \frac{\bm{\xi} \otimes \bm{\xi}}{||\bm{\xi}||^3} \left(\mathbf{u}(\mathbf{y}) - \mathbf{u}(\mathbf{x}) \right)\, d\mathbf{y} + \mathbf{f}_b(\mathbf{x}),
\end{equation}
where $\rho$ is density, $\mathbf{f}_b$ is a body force density, $\bm{\xi}$ is the vector between two material points in the reference configuration, $\mathbf{u}(\mathbf{x})$ is the displacement of $\mathbf{x}$ in the deformed configuration, and $c$ is a material constant that is related to the kernel support $\delta$ (refered to in peridynamics literature as the horizon) and the bulk modulus $\kappa$ via \cite{emmrich2007analysis}
\begin{equation}\label{eq:scaling}
c = 
\begin{cases}
\frac{72 \kappa}{5 \pi \delta^3} \qquad d = 2\\
\frac{18 \kappa}{ \pi \delta^4} \qquad d = 3
\end{cases}.
\end{equation}

We consider equilibrium solutions for which the static problem may be written in the form of \eqref{eq:generalCase} as
\begin{equation}\label{eq:nonl}
\mathcal{L}_\delta [\mathbf{u}](\mathbf{x}) = -\int_{B(\mathbf{x},\delta)} c \frac{\bm{\xi} \otimes \bm{\xi}}{||\bm{\xi}||^3} \left(\mathbf{u}(\mathbf{y}) - \mathbf{u}(\mathbf{x}) \right)\, d\mathbf{y} = f(\mathbf{x}).
\end{equation}

We focus on this model due to the fact that the kernel may be evaluated analytically for the case of displacements $\bm{u}$ belonging to the space of $m^{th}$ order vector polynomials, $P_m(\mathbb{R}^d)$. While this simplifies our presentation, the technique presented in this work generalizes to any non-local operator whose radial kernel has compact support $\delta$. Practically, this model is conceptually straightforward, but presents sufficient complexity to demonstrate the ease with which our approach may be integrated into a standard engineering workflow to study fracture mechanics problems.

Peridynamic models are popular due to their ability to describe the motion of solids while requiring only that $\bm{u} \in \mathbf{L}^2(\Omega)$ \cite{mengesha_du_2014}, in comparison to local models that typically require $\bm{u} \in \mathbf{H}^1(\Omega)$. This allows them to be applied to problems in fracture mechanics, for which the displacement field is discontinuous at cracks. Particle discretizations of peridynamic models are particularly appealing due to their ability to robustly handle large deformations and topological changes occurring during violent fracture without the need to maintain a high-quality mesh. For both mesh-based and meshless discretizations of peridynamic models, the development of efficient and accurate quadrature rules for the integral term poses a significant challenge. In the mesh-based setting, to maintain optimal second order accuracy, the integral of finite element shape functions must be calculated over the intersection of their support and the support of the peridynamic horizon \cite{seleson2016convergence}. Due to the computational expense of these geometric calculations, low order approximations are often used in practice that limit the accuracy of the method.

A key desirable property for any discretization of peridynamics is that it be faithful to the local continuum theory, in the case where the displacement possesses sufficient regularity that the local solution exists. As the peridynamic horizon $\delta$ limits to zero, a peridynamic model reduces to its local counterpart, expressed as a familiar PDE. For the bond-based model considered here this corresponds to a linearly elastic material with bulk modulus $\kappa$ and Poisson ratio $\nu = \frac14$ \cite{emmrich2007well}. A discrete peridynamic model is characterized by two different length scale parameters: the horizon $\delta$ and the characteristic discretization length scale, $h$. It is desirable that in the limit as both of these are simultaneously driven to zero, the solutions of the nonlocal discrete model should recover solutions of the corresponding local continuum model. Discretizations that achieve this without any restrictions on how $\delta, h$ limit to zero are called \emph{asymptotically compatible} discretizations \cite{tian_2014}. In particular, it was shown in \cite{tian_2014} that for piecewise constant approximations, the most commonly used discretization of peridynamics, convergence to the local solution could be obtained if $h = o(\delta)$ as $\delta \rightarrow 0$. In contrast, a piecewise linear discretization was shown to be asymptotically compatible, meaning convergence to the local solution can be demonstrated as $\delta \rightarrow 0$ for a fixed ratio of $\delta/h := M > 0$ \cite{bobaru2009convergence}. Generally, particle methods may be interpreted as piecewise constant discretizations and are therefore susceptible to this issue.

To remedy this limitation for particle discretizations, we seek to generate consistent meshfree quadrature rules of the form
\begin{equation}\label{quadAnsatz}
\mathcal{L}_\delta[u](\mathbf{x}_i) \approx \mathcal{L}^h_\delta[u](\mathbf{x}_i):= \sum_j K(\mathbf{x}_i,\mathbf{x}_j) (u_j-u_i) \omega_{j}
\end{equation}
by using local optimization problems to seek quadrature weights $\omega_{j}$ associated with a local neighborhood of particles that are exact for certain classes of functions. In this way, we are able to obtain an asymptotically compatible discretization that does not depend on the ratio $\frac{h}{\delta}$ to obtain convergence. The application of this quadrature does not require in any way a background grid and is therefore truly meshfree; the quadrature weights are obtained using only a list of neighbors lying within the support of $K$ and, after solving a small local optimization problem requiring only the inversion of a small linear system, weights are obtained which may be used in the standard workflow of an existing peridynamics code.

An often overlooked but important issue with these types of discretizations is their sensitivity to the symmetry of the underlying particle arrangement. In the continuous case, any first-order polynomial vector field $\bm{p}$ lies in the null-space of the operator (i.e. $\mathcal{L}_\delta[\bm{p}]=0$), due to the anti-symmetry of the kernel. For standard discretizations of peridynamics, this necessitates that the quadrature points and weights in Equation \ref{quadAnsatz} be similarly antisymmetric, which limits their applicability to Cartesian particle configurations. A key feature to our approach lacking in existing particle discretizations of peridynamics will be the ability to handle anisotropic particle configurations while maintaining accuracy.

We organize the remainder of the paper as follows. We first present in Section \ref{sec:quad} a definition of a constrained quadratic program (QP) that we use to generate the quadrature weights. We present a proof of the accuracy of this quadrature rule, quantifying the effect of the singularity on the convergence rate of the truncation error. We then outline in Section \ref{sec:peri} how this rule may be specialized to the bond-based peridynamic case. In Section \ref{sec:manu1}, we use manufactured solutions to demonstrate the convergence of the discrete model to both the non-local and local solutions. We use numerical experiments in Section \ref{sec:damage} to demonstrate that with this quadrature rule, a classical bond-breaking damage model recovers a traction-free local boundary condition at crack surfaces as $\delta \rightarrow 0$. In Section \ref{sec:typeIcrack}, we use the approach to recover displacement fields for a type I fracture mode for which an analytic solution to the corresponding local case is available, highlighting the consistent treatment of cracks and traction boundary conditions in the scheme. Finally, in \ref{sec:KWexperiment} we use the scheme to reproduce the Kalthoff-Winkler fracture experiment, demonstrating the applicability of this approach to realistic engineering applications involving dynamic fracture.

\section{Meshfree quadrature rule and analysis}\label{sec:quad}
In this section we present the quadrature rule in a form which may be applied to general non-local models beyond peridynamics. We then provide a proof of its accuracy, from which we may see that an asymptotically compatible discretization is obtained if the degree of polynomial reproduction is greater than the order of singularity in the non-local kernel.
\subsection{Optimization problem}
We define a collection of quadrature points $\mathbf{X}_q = \left\{\mathbf{x}_i\right\}_{i=1,\dots,N_q} \subset B(\mathbf{x},\delta)$ and seek a quadrature rule for a function $f \in \mathbf{V}$ of the following form.
\begin{equation}\label{quadme}
 \int_{B(\mathbf{x})} f\, d\mathbf{y} = \sum_{j \in \mathbf{X}_q} f_j\, \omega_j
\end{equation}
where $\mathbf{V}$ is a function space, and $\left\{\omega_j\right\}$ are a collection of weights determined to exactly reproduce a finite dimensional subspace $\mathbf{V}_h \subset \mathbf{V}$. Informally, this resembles the construction of Gauss quadrature rules, in which quadrature points and weights are selected to exactly reproduce polynomials of maximal degree via interpolation conditions. However, in this case it is assumed that there is no control over the location of the points, as they may be evolving in a Lagrangian manner under the physics of the problem. Further, interpolation in higher than one dimension of unstructured data is in general not possible (see, e.g. \cite{wendland2004scattered}), and so we instead seek to establish the weights via the following least squares problem subject to equality constraints to enforce reproduction:
\begin{align}\label{eq:quadQP}
  \underset{\left\{\omega_j\right\}}{min} \sum_j \omega_j^2 \quad
  \text{such that}, \quad
  I_h[p] = I[p] \quad \forall p \in \mathbf{V}_h
\end{align}
where $I[u]$ and $I_h[u]$ denote the quadrature functional and its approximation (Equation \ref{quadme}). For example, for accuracy reasons one might select $\mathbf{V}_h$ as the space of $m^{th}$-order polynomials $P_m$. Due to our interest in nonlocal applications, we will select $\mathbf{V}_h = P_m \cup S_{K,n,\mathbf{x}}$, where

\begin{equation}\label{eq:reproducingspace}
  S_{K,n,\mathbf{x}} := \left\{ K(\mathbf{x},\mathbf{y}) f(\mathbf{y}) \,|\, f \in P_n \right\}
\end{equation}
and therefore we seek quadrature rules which may reproduce exactly both $m^{th}$ order polynomials and $n^{th}$ order polynomials integrated against a given, possibly singular, kernel $K$.

Quadrature weights may be obtained from Equation \ref{eq:quadQP} by solving the saddle-point problem
\begin{equation}\label{eqn:QPsaddle}
\begin{bmatrix}
\mathbf{I} & \mathbf{B}^\intercal \\
\mathbf{B} & \mathbf{0}
\end{bmatrix}
\begin{bmatrix}
\bm{\omega} \\
\bm{\lambda}
\end{bmatrix}
 = 
 \begin{bmatrix}
 \mathbf{0} \\
 \mathbf{g}
 \end{bmatrix},
\end{equation}
where $\mathbf{I}\in\mathbb{R}^{N_q\times N_q}$ is the identity matrix, $\bm{\omega} \in \mathbb{R}^{N_q}$ consists of the quadrature weights, $\bm{\lambda} \in \mathbb{R}^{dim(\mathbf{V}_h)}$ are a set of Lagrange multipliers used to enforce reproducability, $\mathbf{B} \in \mathbb{R}^{N_q \times dim(\mathbf{V}_h)}$ consists of the reproducing set evaluated at each quadrature point (i.e. $B_{\alpha j} = p^\alpha(\mathbf{x}_j), \, \text{for all } p^\alpha \in \mathbf{V}_h$), and $\mathbf{g} \in \mathbb{R}^{dim(\mathbf{V}_h)}$ consists of the integral of each function in the reproducing set over the ball (i.e. $g^\alpha = I[p^\alpha]$). Equation \ref{eqn:QPsaddle} may easily be solved using a direct solver. We stress that in contrast to mesh-based quadrature rules that require the calculation of intersections between elements and $B(\mathbf{x},\delta)$, the solution of this problem maps easily onto modern computer archtectures for which high-performance dense linear algebra libraries are readily available. In particular, by eliminating the constraints, the problem may efficiently be obtained by solving
\begin{equation}
\bm{\omega} = \mathbf{B}^\intercal \mathbf{S}^{-1}\mathbf{g}
\end{equation}
where $\mathbf{S} = \mathbf{B} \mathbf{B}^\intercal$ is the Schur complement. Therefore, the weights may be obtained in $O(dim(\mathbf{V}_h)^3)$ operations using a direct solver. For problems where the reproducing constraints are redundant, $\mathbf{S}^{-1}$ may be replaced by the pseudoinverse.

While the accuracy of this problem is independent of the choice of basis for $\mathbf{V}_h$, the conditioning of the resulting system may be impacted by a poor choice of basis. For the purposes of this work we express $\mathbf{P_m}$ in terms of the scaled and shifted Taylor monomials $p^\alpha = \frac{1}{\alpha!}\left(\frac{\mathbf{y}-\mathbf{x}}{\delta}\right)^\alpha$.

The generation of quadrature weights satisfying Equation \ref{eqn:QPsaddle} requires that the integral be calculated exactly for each member of $\mathbf{V}_h$ to generate the vector $\mathbf{g}$. For some choices of $\mathbf{V}_h$, this may be done analytically. For example, if polynomial consistency is desired without enforcing exactness of a non-local operator (i.e. selecting $\mathbf{V}_h = \mathbf{P}_m$), exact formulas exist to arbitrary order for the moments of polynomials over balls in arbitrary dimension \cite{folland2001integrate}. For the peridynamic kernels considered here, $\mathbf{g}$ may be calculated analytically using a straightforward change to polar/spherical coordinates. However, we point out for the sake of generality that integrals lacking a closed form solution may be solved using a numerical quadrature rule to integrate polynomials on the unit ball (with e.g. Gauss quadrature) to obtain $\mathbf{g}$, which may then be used to obtain the quadrature weights for the unstructured quadrature points $\mathbf{X}_p$.

\subsection{Analysis of quadrature error}\label{sec:approx}
To consider the impact of the singularity on the accuracy of the quadrature rule, we consider that the singularity in the kernel may be expressed in terms of a bounded numerator and an $\alpha^{th}$ order singularity at the center of the ball, and estimate the reduction in convergence rate caused by the singularity. We consider in our analysis a general non-local operator of the form
$$ \mathcal{L}_\delta[u] = \int_{B(\mathbf{x},\delta)} K(\mathbf{x},\mathbf{y}) u(\mathbf{y}) d\mathbf{y} $$
and afterward specialize to the peridynamic case. We note that for operators with a constant solution in their null-space (i.e. $\mathcal{L}_\delta[1] = 0$) it may be easily shown that this more general form encompasses that used in Equation \ref{eq:generalCase}.
\begin{thm}
  Consider for fixed $\mathbf{x}$ a kernel of the form $K(\mathbf{x},\mathbf{y}) = \frac{n(\mathbf{x},\mathbf{y})}{|\mathbf{y}-\mathbf{x}|^\alpha}$, where the numerator satisfies $n(\mathbf{x},\mathbf{y}) \leq C_n$ for all $\mathbf{y} \in B(\mathbf{x},\delta)$. A set of quadrature weights obtained from Equation \ref{eq:quadQP} with the choice of $\mathbf{V}_h = P_m \cup S_{K,n,\mathbf{x}}$ for $u \in C^{p}$, $p>n$ and ${m}>0$ satisfies the following pointwise error estimate, with $C>0$ independent of the particle arrangement.

  $$\left|\int_{B(\mathbf{x},\delta)} K(\mathbf{x},\mathbf{y}) u(\mathbf{y})\,d\mathbf{y} - \sum_{j \in X_q} K(\mathbf{x},\mathbf{x}_j) u_j \omega_j \right| \leq C \delta^{n+1-\alpha+d}$$
\end{thm}
\begin{proof}
  Let $I[u](\mathbf{x}) = \int_{B(\mathbf{x},\delta)} K(\mathbf{x},\mathbf{y}) u(\mathbf{y}) d\mathbf{y}$ and $I_h[u] =  \sum_{j \in X_q} K(\mathbf{x},\mathbf{x}_j) u_j \omega_j$. Let $p_n$ denote the $n^{th}$-order truncated Taylor series of $u$ about $\mathbf{x}$ with associated remainder $r_n$, such that  
  \begin{align*}
    u(\mathbf{y}) &= p_n(\mathbf{y}) + r_n(\mathbf{y})\\
    &= \underset{|\alpha| {\leq} n}{\sum} \frac{D^\alpha u(\mathbf{y})}{\alpha!}\left(\mathbf{y}-\mathbf{x}\right)^\alpha + \underset{|\beta| = n+1}{\sum} R_\beta(\mathbf{y}) \left(\mathbf{y}-\mathbf{x}\right)^\beta,
  \end{align*}
  where
  $$R_\beta = \frac{|\beta|}{\beta!} \int_0^1 (1 - t)^{|\beta| - 1} D^\beta f\left(\mathbf{y} + t(\mathbf{y}-\mathbf{x})\right) dt.$$
   We now seek an expression of the approximation error in terms of its radial distance from the singularity. We define the unit vector $\mathbf{e}= \frac{\mathbf{y}-\mathbf{x}}{|\mathbf{y}-\mathbf{x}|}$, so that $\mathbf{y}-\mathbf{x} = |\mathbf{y}-\mathbf{x}| \mathbf{e}$. We may then write
   \begin{align*}
     |u-p_n| (\mathbf{y}) = |r_n|(\mathbf{y}) &= \left| \underset{|\beta| = n+1}{\sum} R_\beta(\mathbf{y}) (\mathbf{y}-\mathbf{x})^\beta\right|\\
     & \leq  \underset{|\beta| = n+1}{\sum} | R_\beta(\mathbf{y}) |\, |\mathbf{y}-\mathbf{x}|^{|\beta|}\,| \mathbf{e}^\beta|\\
       &\leq |\mathbf{y}-\mathbf{x}|^{n+1} \underset{|\beta| = n+1}{\sum} | R_\beta(\mathbf{y})|,
   \end{align*}
   so that
   $$|R_\beta(\mathbf{y})| \leq \frac{1}{\beta!} \underset{|\alpha| = |\beta|}{\max} \underset{\mathbf{y} \in B(\mathbf{x},\delta)}{\max} |D^\alpha u(\mathbf{y})|,$$
   and we obtain the bound on the ball, for $C_b>0$,
   $$|u-p_n|({\mathbf{y}}) \leq  C_b |\mathbf{y}-\mathbf{x}|^{n+1}.$$
   To bound the approximation error, we apply the triangle inequality
  \begin{equation*}
    |I[u] - I_h[u]| \leq  |I[u] - I[p_n]| +  |I[p_n] - I_h[u]|
  \end{equation*}
  and obtain from the reproducing condition {of polynomial $p_n$}
  \begin{equation*}
    |I[u] - I_h[u]| \leq  |I[u] - I[p_n]| +  |I_h[p_n] - I_h[u]|
  \end{equation*}
  We now bound these two terms quantifying the continuous and discrete errors separately. The first term may be bound as follows.
   \begin{align*}
     |I[u] - I[p_n]| \leq& \int_{B(\mathbf{x},\delta)} |K(\mathbf{x},\mathbf{y})|\, |u(\mathbf{y})-p_n(\mathbf{y})| d\mathbf{y}\\
     \leq& C_b \int_{B(\mathbf{x},\delta)} \left| \frac{n(\mathbf{x},\mathbf{y})}{|\mathbf{y}-\mathbf{x}|^\alpha} \right| \, |\mathbf{y}-\mathbf{x}|^{n+1} d\mathbf{y}\\
     \leq& C_b C_n \int_{B(\mathbf{x},\delta)} |\mathbf{y}-\mathbf{x}|^{n+1-\alpha} d\mathbf{y}\\
     \leq& C_b C_n \delta^{n+1-\alpha + d}
   \end{align*}
   For the second term in the error estimate:
    \begin{align*}
      |I_h[p_n] - I_h[u]| &\leq \sum_{j \in X_Q} | K(\mathbf{x},\mathbf{x}_j) |\, |u_j-p_n(\mathbf{x}_j)| \omega_j\\
      &\leq \sum_{j \in X_Q} \frac{| n(\mathbf{x},\mathbf{x}_j) |\, |u_j-p_n(\mathbf{x}_j)|}{|\mathbf{x}_j-\mathbf{x}|^\alpha} \omega_j\\
      &\leq C_n C_b \sum_{j \in X_Q} |\mathbf{x}_j-\mathbf{x}_i|^{n+1-\alpha} \omega_j \\
      &\leq C_n C_b \delta^{n+1-\alpha} |B(\mathbf{x}_\delta)|\\
      &\leq C_n C_b \delta^{n+1-\alpha+d}
    \end{align*}
    and therefore we obtain by taking $C = C_b C_n$
    $$|I[u] - I_h[u]| \leq C \delta^{n+1-\alpha+d}.$$
\end{proof}

\begin{remark}
Without loss of generality, we may apply this scalar error estimate component-wise to the peridynamic model. In light of the scaling of the material constants in Equation \ref{eq:scaling} as $O(\delta^{-(d+1)})$, we expect a convergence rate of $O(\delta^{n-\alpha})$ %
when this quadrature rule is applied to the bond-based model. For the remainder of this work, we will take as a reproducing space the minimal reproducing set such that the above proof still holds, namely
$$\mathbf{V}_h = P_0 \cup S_{K,n,\mathbf{x}}.$$

\end{remark}
\begin{remark}
This approach is asymptotically compatible in the sense that, to maintain unisolvency as the horizon size is reduced, the number of points within the horizon must be maintained as roughly constant while the horizon goes to zero. As we will discuss in the following section, for a quasi-uniform particle distribution this corresponds to $\delta = M\, h$ for a fixed $M>0$. Then, we see from the analysis that the discrete non-local operator converges to the continuous one as $\delta \rightarrow 0$. Since the non-local theory converges to the local theory for a suitable choice of kernel scaling, we may see that the nonlocal discrete solution converges to the local continuous solution as $(h,\delta)\rightarrow 0$ at the same rate, provided that the reproducing space is of higher order than the kernel singularity ($n>\alpha$).
\end{remark}
\begin{remark}
The quadrature weights obtained from the constrained QP are tied specifically to the location $\mathbf{x}$ and the quadrature points within $B(\mathbf{x},\delta)$. We will explicitly denote this dependence of the quadrature rule on the location by denoting $\left\{\omega_{j,i}\right\}_{j=1,\dots,N_q}$ as the set of weights associated with the ball centered at $\mathbf{x}_i$. Consider now a point $\mathbf{x}_k$ lying within the horizon of two points $\mathbf{x}_i$ and $\mathbf{x}_j$ (i.e. $|\mathbf{x}_k-\mathbf{x}_i| \leq \delta$ and $|\mathbf{x}_k-\mathbf{x}_j| \leq \delta$). There is no guarantee that $\omega_{k,i} = \omega_{k,j}$ and as such we cannot expect that the current approach will exactly conserve linear momentum when discretizing the bond-based model. However, since the method will converge to the continuous solution, we may expect that energy will be conserved in the limit as the model is refined. While some particle discretizations of peridynamics \cite{parks2008implementing} preserve linear momentum, this comes at the expense of an inconsistent discretization. It is problem dependent whether it is more advantageous to conserve momentum or maintain a consistent method. We note however that this trade-off between consistency and conservation is characterisic of many meshfree methods and we illustrate in the remainder of this paper that high fidelity results may be obtained in a particle-based modelling framework without formal conservation.
\end{remark}

\section{Numerical setting}\label{sec:peri}

We now discretize the peridynamic model over our domain $\Omega$ with a collection of points $\mathbf{X}_h = \left\{ \mathbf{\mathbf{x}}_i \right\}_{i = 1,\dots, N_p} \subset \Omega^\delta$, where 
\begin{equation}
  h = \underset{\mathbf{\mathbf{x}} \in \Omega^\delta}{\sup} \underset{1 \leq j \leq N_p}{\min} || \mathbf{\mathbf{x}} - \mathbf{\mathbf{x}}_j ||_2
\end{equation}
is a length scale characterizing the resolution of the point cloud.
This denotes the largest ball that may be inserted into $\Omega^\delta$ without intersecting another point in $\mathbf{X}_h$. We define the separation distance
\begin{equation}
  q_{\mathbf{X}_h} = \frac12 \underset{i \neq j}{\min} ||\mathbf{\mathbf{x}}_i - \mathbf{\mathbf{x}}_j||_2,
\end{equation}
and refer to $\mathbf{X}_h$ as quasi-uniform if there exists a constant $c_{qu} > 0$ such that
\begin{equation}
q_{\mathbf{X}_h} \leq h \leq c_{qu} q_{\mathbf{X}_h}.
\end{equation}
Such a condition is sufficient to prove polynomial unisolvency for a variety of meshfree techniques (e.g. \cite{wendland2004scattered,mirzaei2011generalized,mirzaei2013direct,trask2017high}), given mild conditions on the geometry. For simplicity, we focus in this work on outlining the new discretization and leave a formal analysis for a future work. We stress here the important connection between the current approach and the depth of approximation theory published regarding generalized moving least squares.

Because we are interested in the asymptotic compatibility of our discretization, for all results presented in this work we fix the horizon as a constant multiple of the particle spacing, $\delta = M h$. Therefore, during refinement, the horizon and the discretization are both taken to zero at the same rate in contrast to many peridynamics discretizations. For simplicity, we will always denote this limit as $\delta \rightarrow 0$ but stress here that this also implies $h \rightarrow 0$ at the same rate.

 To obtain a solution to the optimization problem, it is necessary that the quadrature points within the ball be unisolvent over the desired reproducing set. Informally, this requires that the horizon be set such that there are at least enough neighbors to represent the desired degree of polynomial space. Unless otherwise noted, for all results presented in this work we select $\delta = M\,h$, where $M = m+\frac12$, and $m$ is the order of polynomial reproduction, so that more particles are involved in the quadrature as higher-order reproduction is sought. This may be understood as similar to finite difference methods, for which the stencil size increases as the method moves to higher order.
 
We partition $\mathbf{X}_h$ into \textit{interior points} $\mathbf{X}^I_h = \mathbf{X}_h \cap \Omega$ and \textit{boundary points} $\mathbf{X}^B_h = \mathbf{X}_h \cap \partial^\delta \Omega$. To discretize the IDE at interior points, we apply the following implicit time discretization to Equation \ref{eq:LPSequation} and apply collocation at each point in $\mathbf{X}^I_h$ to obtain

\begin{equation}\label{eq:discreteLPS}
\rho_i \left( \frac{\mathbf{u}^{n+1}_i - 2 \mathbf{u}^n_i + \mathbf{u}^{n-1}_i}{\Delta t^2} \right) = \sum_{j \in B(\mathbf{\mathbf{x}}_i,\delta)} K_{ij} (\mathbf{u}_j^{n+1} - \mathbf{u}_i^{n+1}) \omega_{j,i} + \mathbf{f}_i,
\end{equation}
where $\omega_{j,i}$ are the quadrature weights defined in the previous section and we denote the point evaluation of a function at the $n^{th}$ timestep $f(\mathbf{\mathbf{x}}_i,t^n) = f_i^{n}$. In this way, we obtain one equation for each particle in the domain. In the static case, we obtain the following equation.
\begin{equation}\label{eq:discreteLPSstatic}
-\sum_{j \in B(\mathbf{\mathbf{x}}_i,\delta)} K_{ij} (\mathbf{u}_j - \mathbf{u}_i) \omega_{j,i} = \mathbf{f}_i,
\end{equation}
where we omit the time dependence of $\mathbf{u}$.

To enforce volumetric Dirichlet conditions over boundary points, we simply enforce that for all $\mathbf{\mathbf{x}}_i \in \mathbf{X}_h^B$,
\begin{equation}
  \mathbf{u}_i = \mathbf{u}_i^D
\end{equation}
for a given displacement field $\mathbf{u}^D$. We postpone a discussion of traction-free conditions until Section \ref{sec:damage}.

By requiring that these equations hold for all points, we may assemble a global stiffness matrix for the problem. While a thorough discussion of the optimal solution of these linear systems is beyond the scope of the current work, we note that the lack of symmetry in the quadrature weights causes the resulting stiffness matrices to be asymmetric. Preliminary scaling results using an unsmoothed aggregation AMG preconditioner with a GMRes iterative solver exhibited a roughly $O(N_p^{1.5})$ scaling for the two dimensional results presented in this paper when solving static problems. For the unsteady problems, the diagonal contribution approaches optimal $O(N_p)$ convergence as $\Delta t \rightarrow 0$ and the linear system becomes more diagonally dominant.

\section{Truncation error and manufactured solutions}\label{sec:manu1}

We now use manufactured solutions to confirm both the truncation error analysis of Section \ref{sec:approx} and to confirm that the solution of Equation \ref{eq:discreteLPSstatic} converges to both non-local solutions with high-order and to the local solution in an asymptotically compatible manner. In this section, we refer to exact local and non-local solutions as $\mathbf{u}_{ex}$ and $\mathbf{u}_{ex}^\delta$, respectively and to the numerical solution $\mathbf{u}_{h}$. For both cases we consider convergence on the square $\Omega = \left\{-\pi,\pi\right\}\times\left\{-\pi,\pi\right\}$. We will consider both the truncation error $\epsilon_t = |\mathcal{L}_\delta[\mathbf{u}]-|\mathcal{L}^h_\delta[\mathbf{u}]|$ and the local and non-local solution errors $\epsilon_{s} = |\mathbf{u}_{ex} - \mathbf{u}_h|$ and $\epsilon_{s,\delta} = |\mathbf{u}^\delta_{ex} - \mathbf{u}_h|$. We quantify error in the following $\ell_2$-norm, taken as root-mean-square of the point evaluation of the magnitude of a vector field at each particle, namely
\begin{equation}
||\mathbf{F}||_2 = \sqrt{ \frac{\sum_p |\mathbf{F}|_p }{N_p}  }.
\end{equation}
For all results, the domain is first discretized by a Cartesian lattice with spacing $h$. The particles are then perturbed by a uniformly distributed random variable with magnitude $0.1 h$ to remove any symmetries that might lead to accelerated convergence.
\subsection{Convergence to non-local solution}

We take here the displacement $\mathbf{u} = \left<(1-x)^6+(1-y)^6,0\right>$, from which a symbolic mathematics package may be used to obtain
\begin{align*}
  \mathcal{L}^\delta[\mathbf{u}] =& \left< \left(280 - 840 x + 1260 x^2 -840 x^3 +210 x^4 -280 y +420 y^2 -280 y^3 +70 y^4\right) \right.\\
  + &\left( 126 -210 x +105 x^2 - 42y +21 y^2\right)\delta^2\\
   + & \left. 5 \delta^4, 0 \right> .
\end{align*}
We present in Figure \ref{fig:manu_nonloc} convergence rates for the non-local truncation and manufactured solution errors. We apply increasing orders $n$ for the polynomial reproducing space (i.e. $S_{K,n,\mathbf{x}}$ in Equation \ref{eq:reproducingspace}), using quadratic, cubic, and quartic polynomials. In agreement with the scaling analysis of Section \ref{sec:approx}, we observe $O(\delta^{n-1})$ convergence consistent with the $(\alpha=1)$-order singularity in the peridynamic kernel. We empirically observe $O(\delta^{n})$ convergence in the solution error for even-order polynomial spaces, although increasing to cubic reproduction from quadratic reproduction does not increase the convergence rate. We attribute this behavior to the fact that the integrand of the non-local operator is anti-symmetric on the ball for odd-order polynomial moments. 
\begin{figure}[t!]
    \centering
    \includegraphics[width=0.49\textwidth]{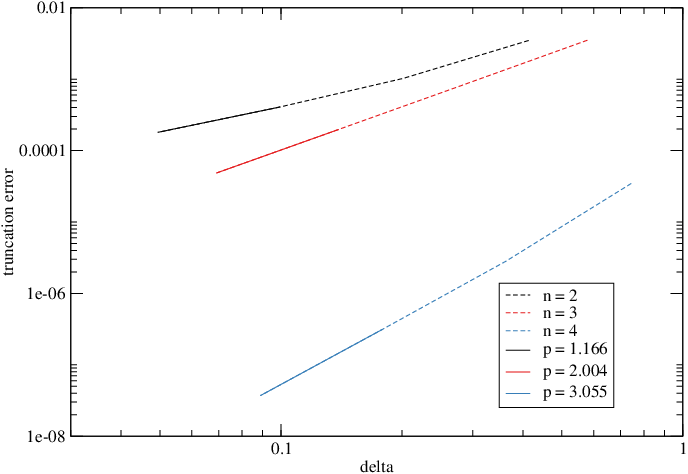}
    \includegraphics[width=0.49\textwidth]{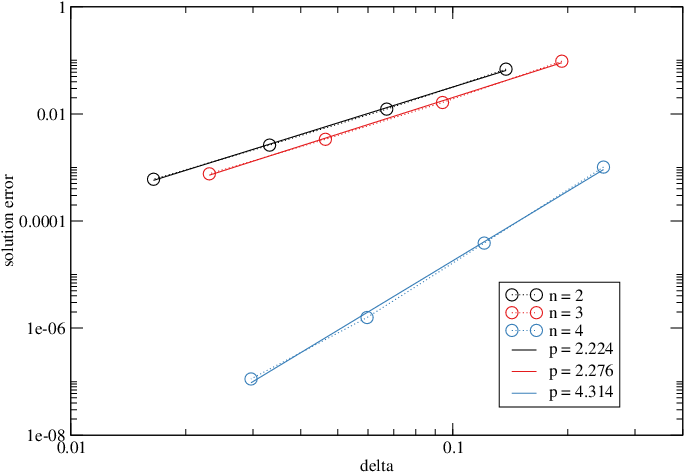}
    \caption{Convergence of truncation error $||\mathcal{L}^h_\delta[\mathbf{u}]-\mathcal{L}_\delta[\mathbf{u}]||_2$ (\textit{left}) and manufactured solution error $||\epsilon_{s}||_2$ (\textit{right}) for non-local case. Power regression of form $||\mathcal{L}^h_\delta[\mathbf{u}]-\mathcal{L}_\delta[\mathbf{u}]||_2 = C h^p$ demonstrates convergence rate in the regime $h << 1$.}
    \label{fig:manu_nonloc}
\end{figure}

\subsection{Convergence to local solution}
We consider now the smooth solution $\mathbf{u} = \left<\sin x \sin y,\cos x \cos y \right>$, which converges as $O(\delta^2)$ to the local elasticity operator
\begin{equation}
\mathcal{L}[\mathbf{u}] = -\frac{6}{5} \left<\sin x \sin y,\cos x \cos y \right>,
\end{equation}
following Emmrich~\cite{emmrich2007analysis} with bulk modulus $\kappa = 1$ and Poisson ratio $\nu = \frac14$. %
In Figure \ref{fig:manu_loc} we observe that the non-local operator converges to the local operator as $O(\delta^2)$ for both the truncation error and manufactured solution error, independent of the choice of reproducing space. This is due to the fact that the non-local theory itself converges to the local as $O(\delta^2)$ independent of the choice of discretization used.
\begin{figure}[t!]\label{fig:manu_loc}
    \centering
    \includegraphics[width=0.49\textwidth]{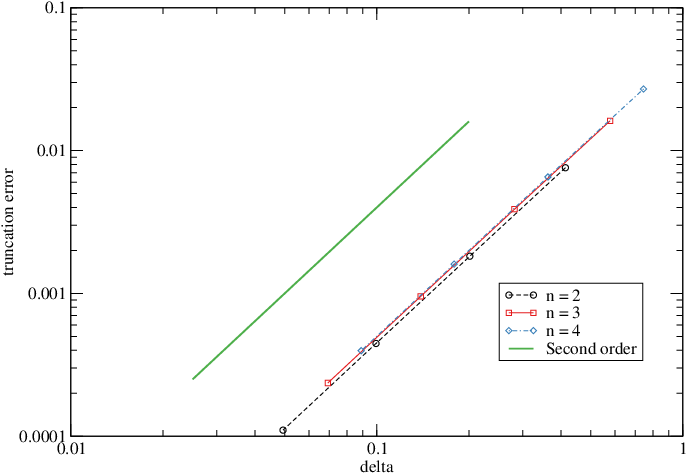}
    \includegraphics[width=0.49\textwidth]{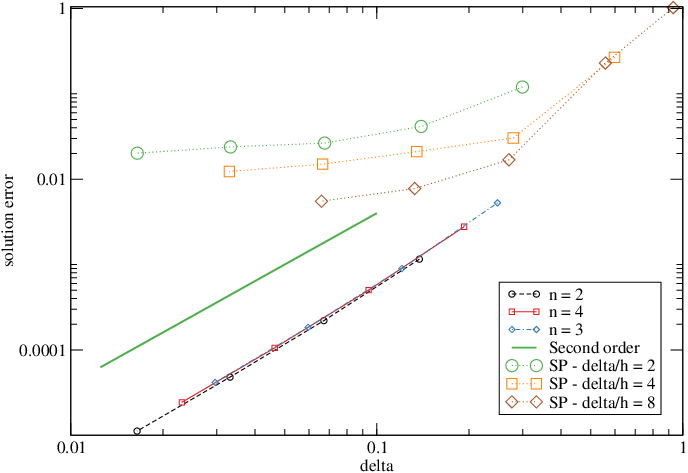}
    \caption{Convergence of truncation error $||\mathcal{L}^h_\delta[u]-\mathcal{L}[u]||_2$ (\textit{left}) and manufactured solution error $||\epsilon_{s,\delta}||_2$ (\textit{right}) for local case. Second order convergence is presented for comparison, along with convergence results for standard peridynamics quadrature (denoted SP).}
\end{figure}

\section{Damage modeling and traction boundary conditions}\label{sec:damage}

In bond-based models for peridynamic fracture, the fracture is typically modeled by the breaking of bonds due to a critical strain criteria \cite{silling_2005_2}. We demonstrate here that the meshfree quadrature is compatible with the typical workflow of a peridynamics simulation. Given the stencil associated with the point $\mathbf{\mathbf{x}}_i$, a bond $B_{ij}$ is associated with each neighbor $\mathbf{x}_j \in B(\mathbf{\mathbf{x}}_i,\delta)$. Each bond is then assigned a state of either being broken or unbroken, and the quadrature rule is modified to read \cite{madenci2014peridynamic}
\begin{equation}
\tilde{\omega}_{j,i} =
  \begin{cases}
    \omega_{j,i}, & \text{ if bond is unbroken }  \\
	0, & \text{ if bond is broken }  .
  \end{cases}
\end{equation}

Defining the strain of a bond as
\begin{equation}
s = \frac{ |\mathbf{u}_j - \mathbf{u}_i| - |\mathbf{x}_j-\mathbf{x}_i|}{|\mathbf{x}_j-\mathbf{x}_i|},
\end{equation}
we adopt for this work the prototype microelastic brittle (PMB) model \cite{silling_2005_2,madenci2014peridynamic} and break at each timestep any bond exceeding a critical stress value $s_0$. This parameter may be tuned to match the fracture energy $G_c$ of a given material via \cite{madenci2014peridynamic}, as

\begin{equation}\label{eq:s0}
s_0 =
  \begin{cases}
    \sqrt{\frac{G_c}{ \left( \frac{6 \mu}{\pi} + \frac{16}{9 \pi^2} \left( K - 2 \mu \right) \right) \delta }}, & \text{ d = 2 }  \\
	\sqrt{\frac{G_c}{ \left( 3 \mu + \left(\frac34\right)^4 \left(K - \frac{5 \mu}{3}\right) \right) \delta }}, & \text{ d = 3 }  .
  \end{cases}
\end{equation}
where $\mu$ is the shear modulus and $\kappa$ is the bulk modulus.

To introduce a crack into the initial configuration of a problem, a preprocessing step may be introduced to break any bond $B_{ij}$ for which the line segment $\overline{\mathbf{x}_i \mathbf{x}_j}$ intersects the crack surface. We have empirically observed that when this process is used in conjunction with the quadrature rule introduced in this work, the resulting solution converges to the corresponding local theory with a traction-free boundary condition imposed at the crack surface, without the need to introduce any corrections for so-called surface effects \cite{le2017surface}. As such, traction-free boundary conditions may be imposed by introducing dummy particles on $\partial^\delta \Omega$ and breaking any bonds that intersect a subset of $\partial \Omega$ upon which zero-traction conditions apply. Such a process is consistent with extrapolatory dummy particle techniques that are popular within a variety of non-local meshfree methods (see for example \cite{macia2011theoretical}). 

Because this approach is specific to bond-based peridynamic models, we opt to illustrate its effectiveness with numerical examples rather than pursue a formal discussion of its accuracy. In the following section, we demonstrate convergence of this approach to the local case using manufactured solutions. We will then apply this approach to realistic engineering applications in Section \ref{sec:applications}, considering the static loading of a crack in Section \ref{sec:typeIcrack} and the fracture dynamics of a violent impact problem in Section \ref{sec:KWexperiment}.

\subsection{Asymptotic compatibility of damage model to local traction-free condition}
We consider again the box $\Omega = \left\{-\pi,\pi\right\}\times\left\{-\pi,\pi\right\}$. We now introduce a crack along the $x=0$ plane with normal $\mathbf{n} = \left<1,0 \right>$ where we will apply a traction-free condition, namely that $\sigma \cdot \mathbf{n} = 0$, where $\sigma$ is the local linearly elastic stress tensor. For the choice of material parameters $K = 1$ and $\nu = \frac14$ it may be shown that the displacement field $\mathbf{u}_{tf} = \left<x+y,-x-3y)\right>$ satisfies these conditions (i.e. $\mathcal{L}^\delta[\mathbf{u}_{tf}] = \mathcal{L}[\mathbf{u}_{tf}] = 0$). This case thus provides an extension to the standard patch test, where the discretization should be able to reproduce linear displacements while recovering a traction-free condition along the $x=0$ plane as $\delta \rightarrow 0$.

To discretize this case with the bond-breaking model, we solve the manufactured problem as in Section \ref{sec:manu1}, by imposing $\mathbf{u}_{tf}$ as a Dirichlet condition on $\Omega^\delta$, but now apply the crack model and break any bond crossing the $x=0$ plane. As in the previous cases, we seek to demonstrate that as $\delta \rightarrow 0$, we recover the local solution, and in doing show that the damage model recovers the local theory as the horizon tends to zero. To post-process these results, we generate the Delaunay triangulation associated with our particles and plot the pointwise evaluation of x-component of error normal to the crack along the $y=0$ line (i.e. $(\mathbf{u}_{tf}-\mathbf{u}^\delta_h)(x,0)$) in Figure \ref{fig:manu_crack}. We see that in both the $\ell_2$-norm and the sup-norm (i.e. $||f||_\infty = \underset{\mathbf{x} \in \Omega}{\max} f(\mathbf{x})$) the discrete non-local solution converges to the local solution as $O(\delta)$.

\begin{figure}[t!]
    \centering
    \includegraphics[width=0.49\textwidth]{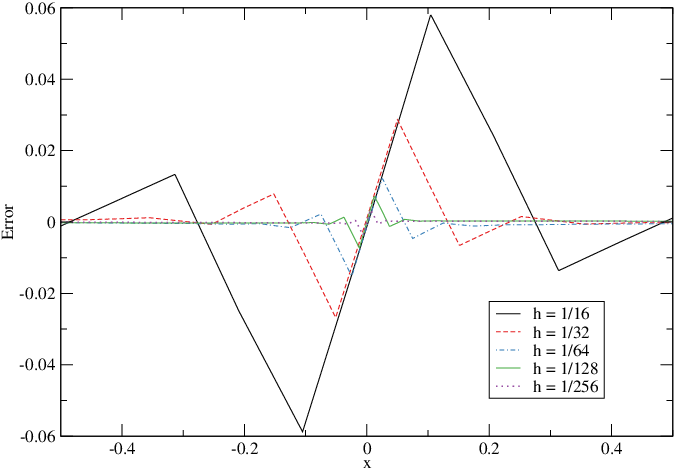}
    \includegraphics[width=0.49\textwidth]{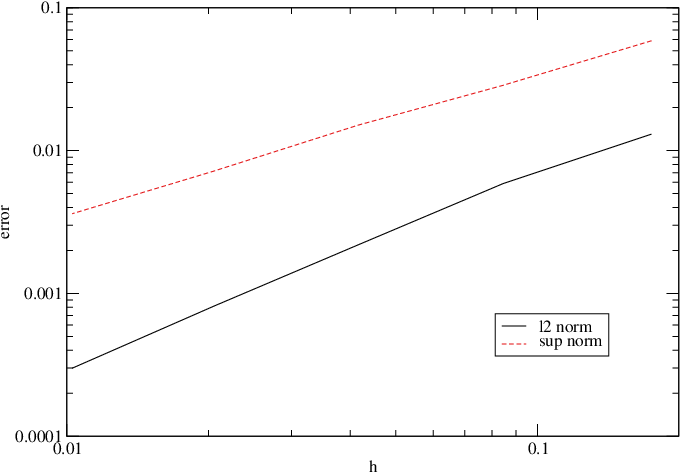}
    \caption{Pointwise convergence of $u^\delta_h \rightarrow u_{tf}$ along the line $y = 0$ demonstrating asymptotic compatibility of crack model (\textit{left}) and comparison of $\ell_2$ and sup-norm (\textit{right}) for the manufactured crack case.}
\label{fig:manu_crack}
\end{figure}

\section{Applications}\label{sec:applications}
Having demonstrated the asymptotic compatibility of both the discretization and the damage model to the local theory, we now present examples using our approach to consider practical engineering problems. In Section \ref{sec:typeIcrack} we simulate the static loading of a Type-I crack for which displacements satisfying the local theory are available in closed form. In Section \ref{sec:KWexperiment} we consider the Kalthoff-Winkler experiment, wherein the fracture dynamics driven by an impactor striking a pre-notched place generates an experimentally reproducable crack pattern. These two cases demonstrate the ability of the discretization to resolve both static and dynamic problems. For both cases, a stress singularity exists at the tip of the cracks that poses challenges for some discretizations.

\subsection{Static loading: Type I crack problem}\label{sec:typeIcrack}

We consider now the Type-I crack problem, in which a crack of length $2a$ is subjected to biaxial loading in the infinite plane, as illustrated in Figure \ref{fig:type1crack}. The displacement field $\mathbf{u} = \left<u,v\right>$ for this problem is
\begin{align}
2\mu u &= \frac{\kappa-1}{2} \sigma_0 \sqrt{r_1r_2} \cos(\frac{\theta_1+\theta_2}{2}) - \frac{\sigma_0 r^2}{\sqrt{r_1r_2}}\sin \theta \sin (\theta - \frac12 \theta_1 -\frac12 \theta_2) \label{eqn:6.1}\\
2\mu v &= \frac{\kappa+1}{2} \sigma_0 \sqrt{r_1r_2}\sin(\frac{\theta_1+\theta_2}{2}) - \frac{\sigma_0 r^2}{\sqrt{r_1r_2}} \sin \theta \cos(\theta - \frac12 \theta_1 - \frac12 \theta_2) ,
\end{align}
where $\mu = \frac{3 \kappa (1-2\nu)}{2(1+\nu)}$, and $\kappa = 3-4\nu$. The coordinates $r$, $r_1$, $r_2$, $\theta_1$ and $\theta_2$ are specified in Figure \ref{fig:type1crack}. The detailed derivation of this analytic solution may be found in \cite{sun2012fracture}.

\begin{figure}[t!]
\centering
\includegraphics[width=0.49\textwidth]{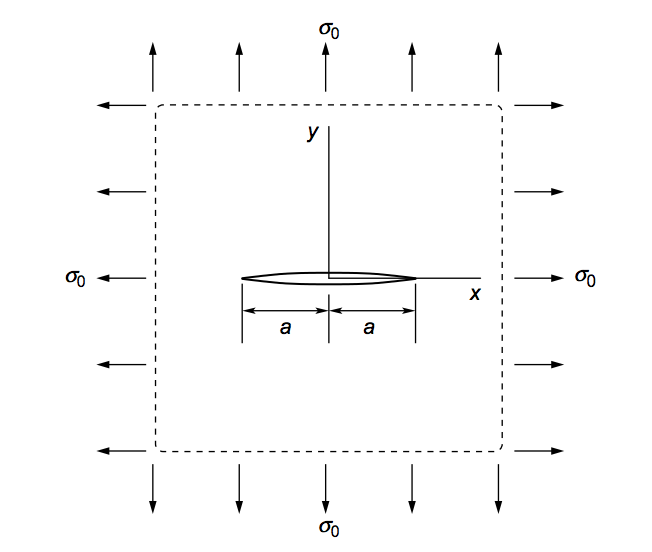}
\includegraphics[width=0.49\textwidth]{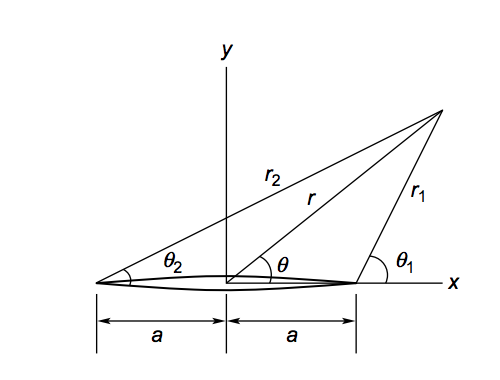}
\caption{Type-I fracture problem geometry (\textit{left}) and coordinate system used to construct analytic solution to displacement field (\textit{right}).}
\label{fig:type1crack}
\end{figure}

To model this problem, we restrict the infinite domain to the box $\Omega = [-2,2] \times [-2,2]$ and impose the analytic solution as a Dirichlet condition around the perimeter. We select as geometric and material parameters: $a = 1$, $\kappa = 4\times 10^4$,$\nu = 1/4$ and $\sigma_0 = 1$. As in the previous example, we break any bond crossing the crack before running the simulation. We run this crack problem with $32^2$, $64^2$ and $128^2$ points, keeping a fixed ratio of $\delta/h = 3.5$. We provide a comparison of the numerical result to the exact solution along the $y=0$ and $x=0$ planes in Figure \ref{fig:disalongaxis}. As expected given results for the manufactured solution in the previous section, the solution reproduces the displacement field and converges as ${\color{red}O(\delta)=O(h)}$.

\begin{figure}[t!] 
\centering
\includegraphics[width=0.49\textwidth]{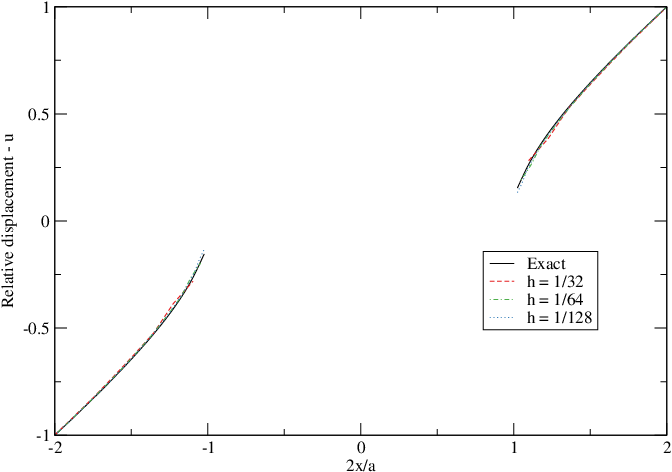}
\includegraphics[width=0.49\textwidth]{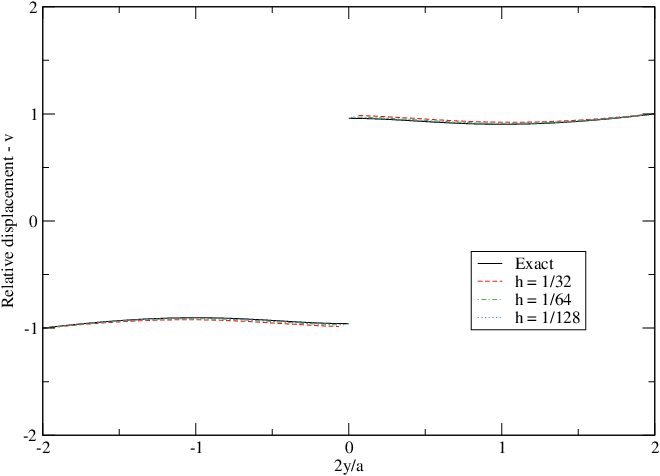}
\caption{Horizontal displacement along $y=0$ axis (\textit{left}) and vertical displacement along $x=0$ axis (\textit{right}). Displacements are scaled by $||u||_{L_\infty}$ and $||v||_{L_\infty}$, respectively.}
\label{fig:disalongaxis}
\end{figure}

\subsection{Dynamic fracture: the Kalthoff-Winkler experiment}\label{sec:KWexperiment}

We finally consider a dynamic fracture problem in which a steel plate is struck by a cylindrical impactor. The plate is pre-notched (Figure \ref{fig:KHsetup}), and upon impact an experimentally reproducable crack grows outward from the pre-notch tip. The fracture pattern behaves differently depending on regimes governed by the impactor velocity. We consider the parameters given in Figure \ref{fig:KHsetup} which match those investigated previously with a particle-based peridynamics discretization \cite{silling2001peridynamic}, for which the fracture is experimentally observed to form a $68^\circ$ angle to the initial vertical pre-notch.

Homogeneous Dirichlet boundary conditions $\mathbf{u}=\left<0,0\right>$ are imposed on the top of the plate to the left and right of the notches, and impact is modeled by imposing a fixed displacement of $\mathbf{u} = \left<0,-32 t\right>$ between the two notches. The remaining portions of the boundary at the sides and bottom are treated as free surfaces following Section \ref{sec:damage}, by breaking any bonds that cross $\overline{\Omega} \cap \overline{\partial^\delta \Omega}$. Equations \ref{eq:discreteLPS} are then evolved forward in time, breaking any bonds that violate the strain criteria at the end of each timestep. We apply the material properties used in \cite{silling2001peridynamic}, for which handbook data provides a yield strength of $2000\,MPa$ and a fracture toughness of $90\,MPa-m^{1/2}$. For these properties, we obtain a bond breaking criterion of $s_0 = 0.0099/\sqrt{\delta}$ following Equation \ref{eq:s0}.

We present results in Figure \ref{fig:KHresult1} using $256\times512$ particles, $\delta = 4.0 h$ and a timestep of size $\Delta t = 2.5e-7$, selected to match the CFL condition imposed by the crack speed. As illustrated in the figure, we see crack growth at an angle of $68^\circ$ to the pre-notch tips, and at the end of the simulation three fragments remain. In what follows, we systematically investigate the influence of the ratio $\delta/h$, the effect of under-resolution, and the impact of under-resolving the timestep $\Delta t$.

In Figure \ref{fig:KHresult2} we repeat the experiment for a range of $\delta/h$ and compare the top-right fragment shape. For reference, the experimentally observed $68^\circ$ angle is superimposed at the notch tip. The extant of damage is visualized by plotting for each particle the fraction of its bonds that are broken. We see that the use of the consistent quadrature rule provides consistent predictions of the fragment shape, independent of the choice of $\delta/h$. This is in contrast to existing work in particle discretizations of peridynamics, where $\delta/h$ is often taken as large to mitigate grid imprinting effects \cite{parks2008implementing}.

We next demonstrate that for $\Delta t = 1.0e-6$ and $\delta/h = 4$ the ability of this approach to reproduce the crack angle when the simulation is under-resolved. We present in Figure \ref{fig:KHresult4} the upper-right fragment when the domain is discretized with $32\times64$, $64\times128$ and $128\times256$ particles. For all but the coarsest resolution we reproduce the $68^\circ$ angle. 

We note that because we solve the governing equations implicitly at each time-step, we are assuming quasi-static loading of the crack whereby the crack propogates at a timescale of $\Delta t$ while the elastic modes propogating faster than $\Delta t$ are effectively damped out. By selecting the timestep based on the CFL condition, we justify this assumption. For comparison however, we compare to results in which $\Delta t$ is selected four and sixteen times larger in Figure \ref{fig:KHresult3}. In this case, fragments of similar size are produced but the extant of the damage near the crack is substantially more diffuse.
\begin{figure}[t!]
    \centering
    \includegraphics[width=0.87\textwidth]{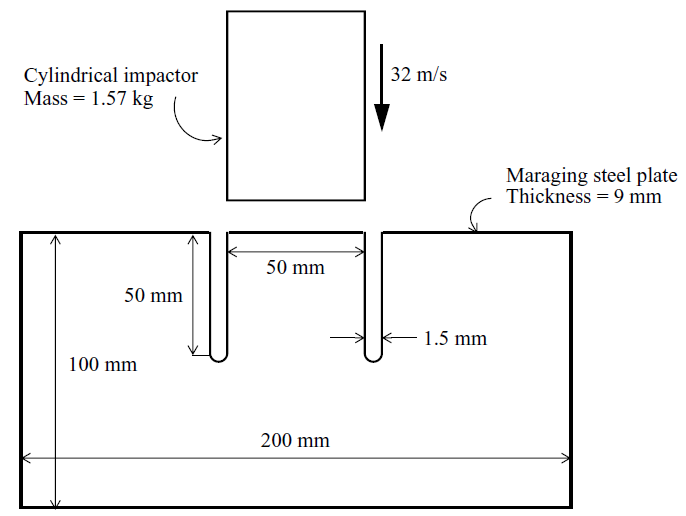}
    \caption{Experimental setup for the Kalthoff-Winkler experiment, taken from \cite{silling2001peridynamic}.}
\label{fig:KHsetup}
\end{figure}

\begin{figure}[t!]
    \centering
    \includegraphics[width=0.47\textwidth]{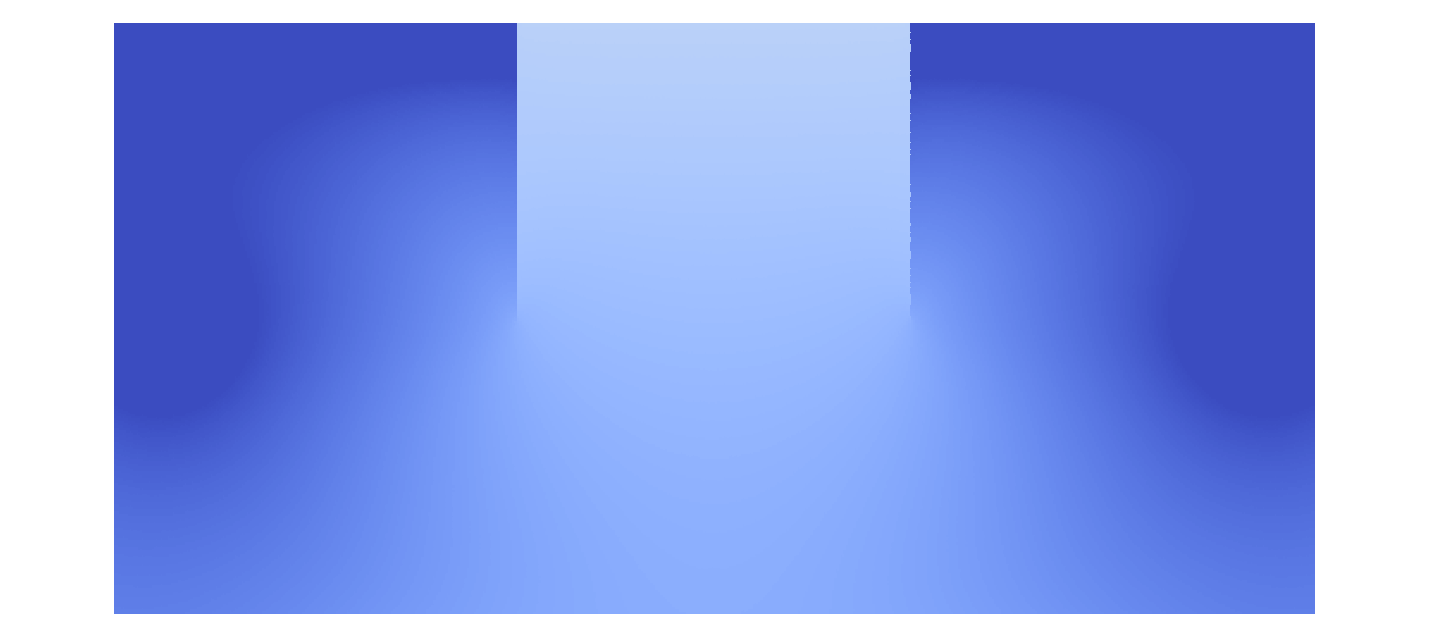}
    \includegraphics[width=0.47\textwidth]{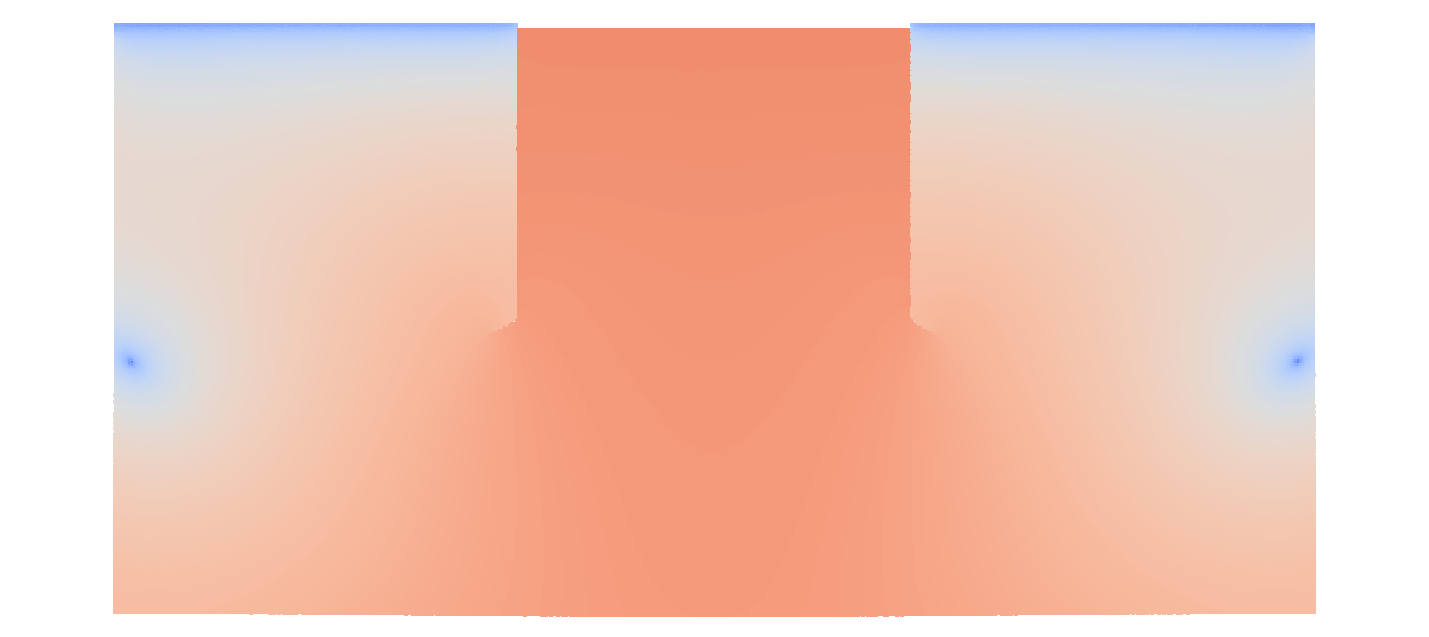}\\
    \includegraphics[width=0.47\textwidth]{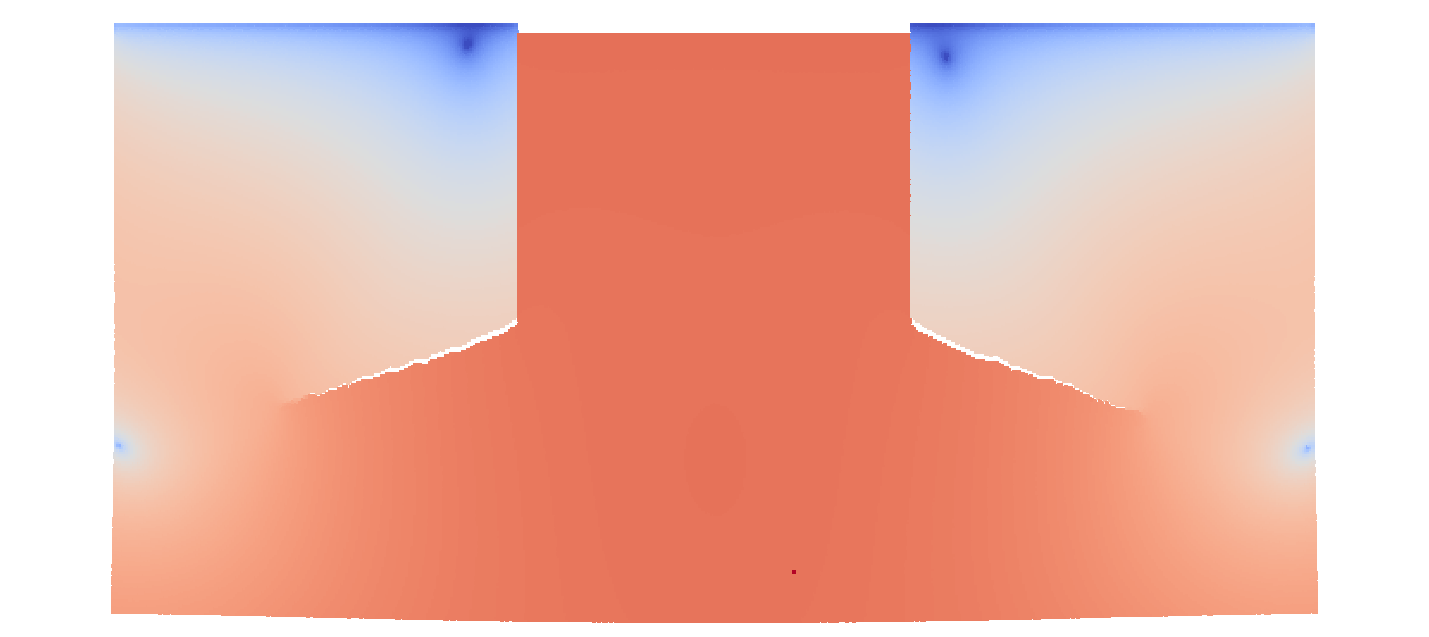}
    \includegraphics[width=0.47\textwidth]{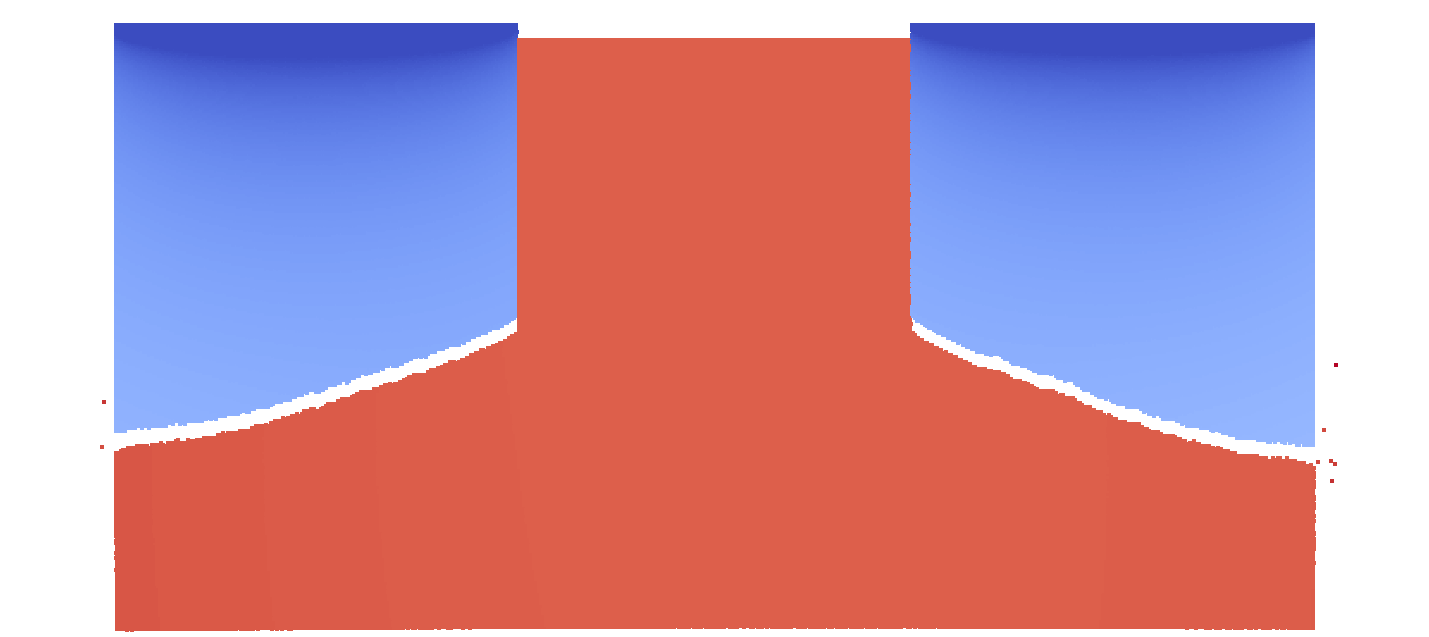}
    \caption{Crack evolution for Kalthoff-Winkler case after $1,100,200$ and $300$ timesteps. Plot is colored by the log of the displacement magnitude ($\log_{10} ||u||).$}
    \label{fig:KHresult1}
\end{figure}

\begin{figure}[t!]
    \centering
    \includegraphics[width=0.23\textwidth]{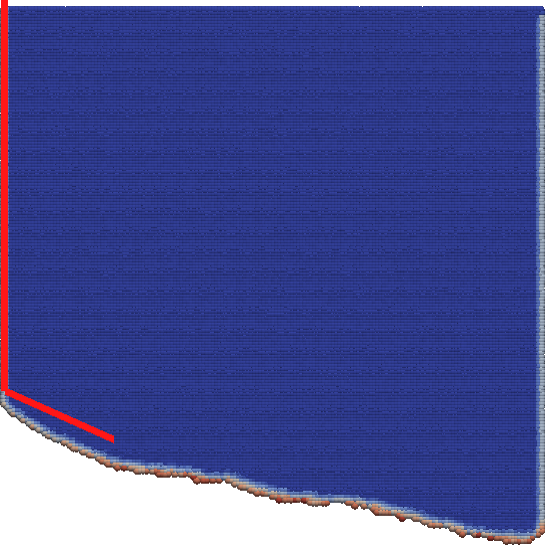}
    \includegraphics[width=0.23\textwidth]{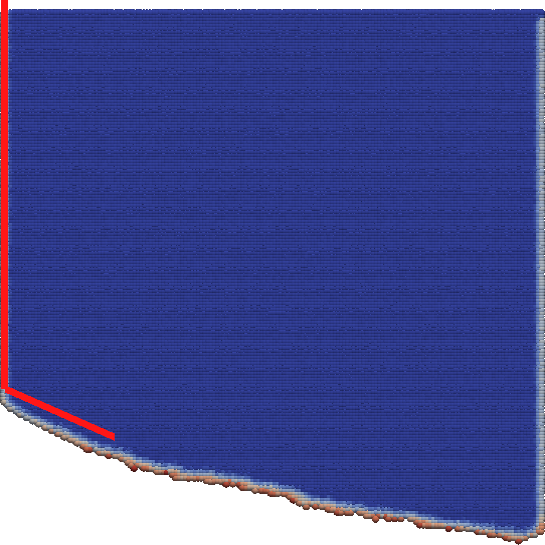}
    \includegraphics[width=0.23\textwidth]{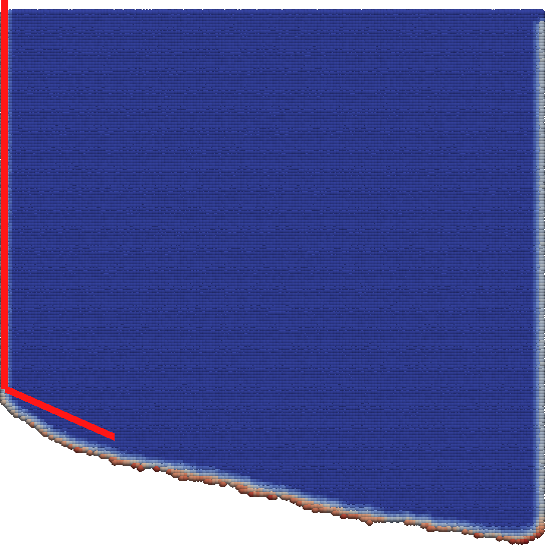}
    \includegraphics[width=0.23\textwidth]{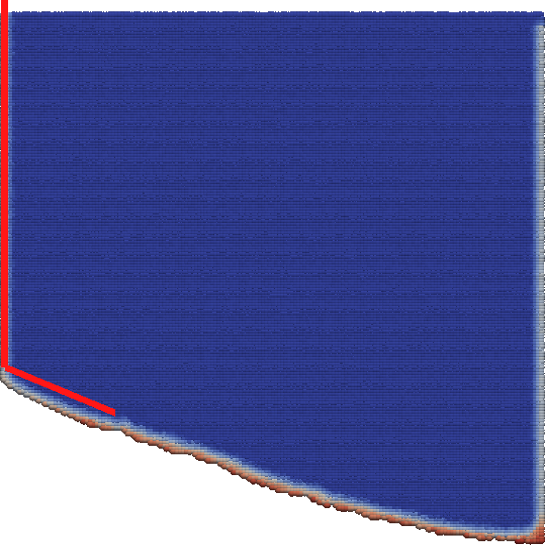}
    \caption{Top-right fragment from Kalthoff-Winkler case for $\delta \in \left\{2.5 h, 3.0 h, 3.5 h, 4.0 h\right\}$. Experimentally observed $68^\circ$ fracture angle at tip of pre-notch is illustrated for comparison.}
    \label{fig:KHresult2}
\end{figure}

\begin{figure}[t!]
    \centering
    \includegraphics[width=0.31\textwidth]{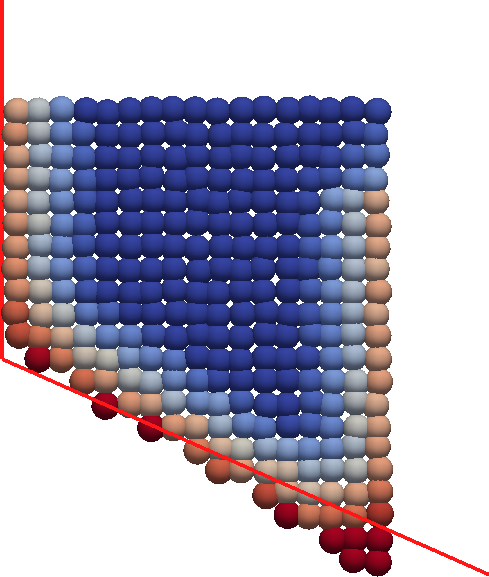}
    \includegraphics[width=0.31\textwidth]{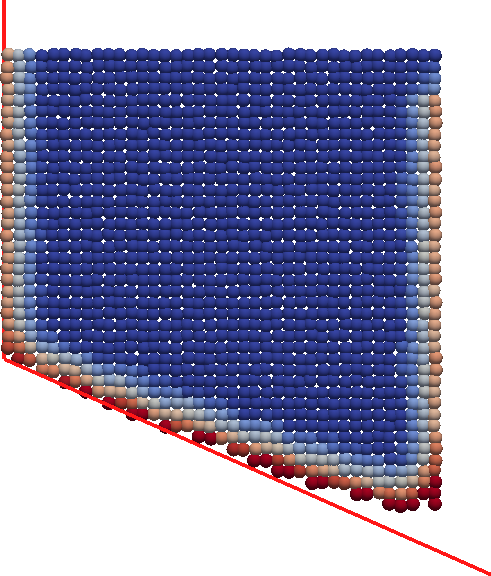}
    \includegraphics[width=0.31\textwidth]{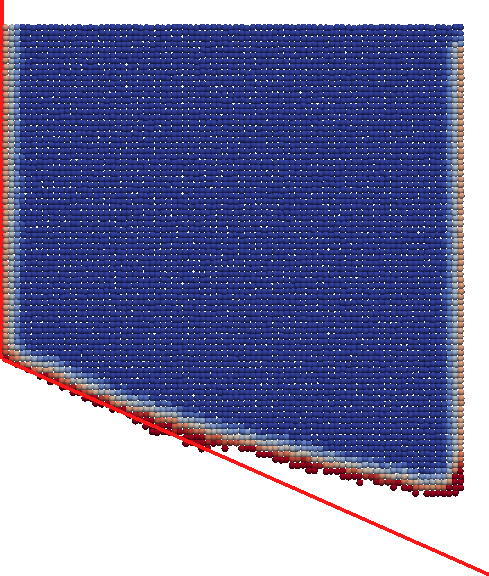}
    \caption{Top-right fragment from Kalthoff-Winkler case for $\delta = 4.0\,h$ with $h \in \left\{1/32, 1/64, 1/128\right\}$. Experimentally observed $68^\circ$ fracture angle at tip of pre-notch is illustrated for comparison.}
    \label{fig:KHresult4}
\end{figure}

\begin{figure}[t!]
    \centering
    \includegraphics[width=0.6\textwidth]{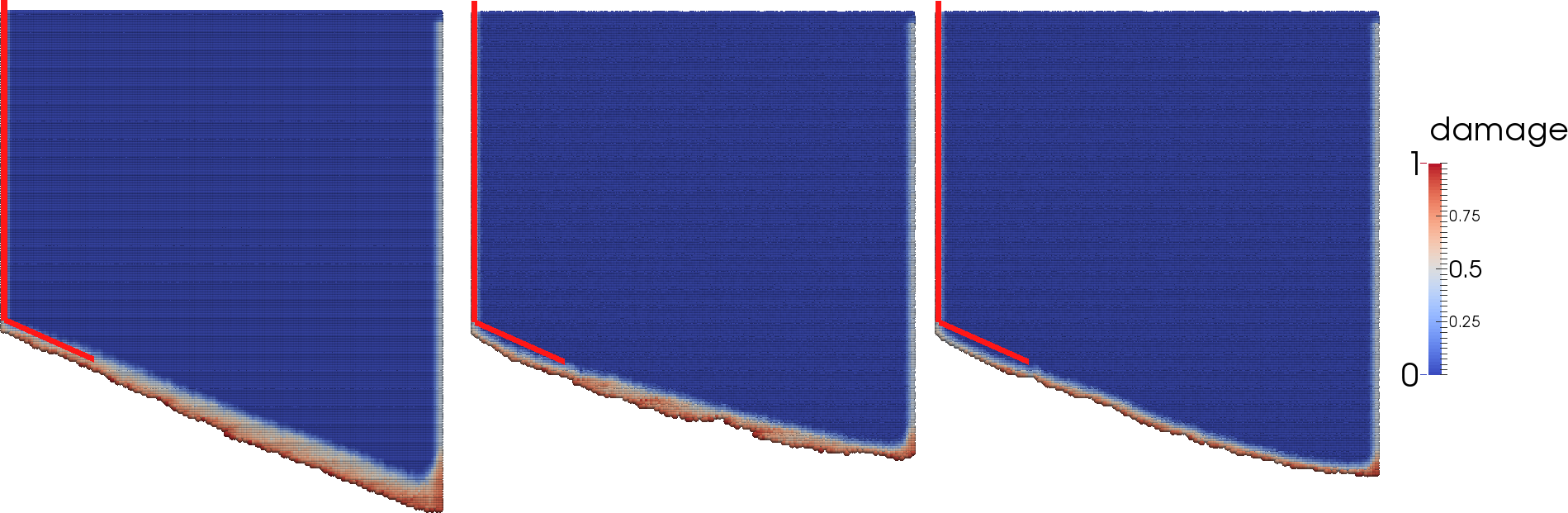}
    \caption{Top-right fragment from Kalthoff-Winkler case for $\delta = 4.0 h$ and $\Delta t = 4.0e-6$ (\textit{left}), $\Delta t = 1.0e-6$ (\textit{center}) and $\Delta t = 2.5e-7$ (\textit{right}). While the larger timestep violates the quasi-static loading assumption, the resulting fragments reproduce crack angle, albeit with more damage.}
    \label{fig:KHresult3}
\end{figure}

\section{Conclusion}

We have introduced a new optimization based meshfree quadrature rule which may be used to generate discretizations of general nonlocal problems, and have shown that this approach is able to provably obtain high-order convergence. We have specifically investigated the suitability of such approaches for generating strong form discretizations of bond-based peridynamic models. For these models, the meshfree generation of quadrature weights fits naturally into existing workflows for such problems, and we have demonstrated the compatibility of this approach with existing damage models. We have demonstrated numerically that for problems with and without fracture, we are able to obtain a particle discretization that obtains first and second order convergence to the local theory, respecitively. We therefore obtain for the first time an asymptotically compatible particle method for peridynamics, both for the PDE solver and for the modeling of fracture.

There are close ties to the optimization problem proposed here and the approximation theory underpinning generalized moving least squares methods. In future work, we will demonstrate rigorously under what conditions a solution to the optimization problem exists. The approach presented here is also applicable to a wide range of nonlocal problems, and we will continue to use this optimization based quadrature as a means to build new non-local particle discretizations built upon a sound approximation theory.

\section*{Acknowledgements}
N. Trask acknowledges support from the National Science Foundation MSPRF program, the Sandia National Laboratories LDRD program, and by the U.S. Department of Energy Office of Science, Office of Advanced Scientific Computing Research, Applied Mathematics program as part of the Colloboratory on Mathematics for Mesoscopic Modeling of Materials (CM4), under Award Number DE-SC0009247. {H. You acknowledges support through the NSF-MSGI program for his work as a summer intern at SNL, and from the National Science Foundation under award DMS 1620434. Y. Yu acknowledges support from National Science Foundation under award DMS 1620434. M.L. Parks acknowledges support from the U.S. Department of Energy Office of Science, Office of Advanced Scientific Computing Research, Applied Mathematics program as part of the Colloboratory on Mathematics for Mesoscopic Modeling of Materials (CM4).} 

\bibliographystyle{unsrt}
\bibliography{yyu}

\end{document}